\documentclass{article}

\usepackage{latexsym,graphicx,xcolor}
\usepackage{amsmath}
\usepackage{amssymb}
\usepackage{float}
\usepackage{amsthm}
\usepackage[english]{babel}
\usepackage[utf8]{inputenc}
\usepackage{authblk}
\usepackage[T1]{fontenc}
\usepackage{algorithm}
\usepackage{algorithmic}
\usepackage{tikz}
\usetikzlibrary{trees}
\usetikzlibrary{positioning}
\usepackage{caption}
\usepackage{cleveref}
\usepackage{enumerate}
\usepackage[normalem]{ulem}

\newtheorem{theorem}{Theorem}
\newtheorem*{theorem*}{Theorem}
\newtheorem{lemma}{Lemma}

\newtheorem{prop}{Proposition}
\newtheorem{corollary}{Corollary}

\newtheorem{remark}{Remark}

\newtheorem{example}{Example}

\newtheorem{conj}{Conjecture}

\newtheorem{claim}{Claim}
\newtheorem*{keywords}{Keywords}

\newcommand{\R}{\mathbb{R}}

\newcommand{\N}{\mathbb{N}}

\newcommand{\D}{\mathcal D}
\newcommand{\MM}{\mathcal M}

\renewcommand{\L}{\mathcal L}
\renewcommand{\P}{\mathcal P}
\renewcommand{\ll}{\lambda_1}

\newcommand{\CC}{\mathcal C}

\newcommand{\rar}{\rightarrow}
\newcommand{\imp}{\Rightarrow}

\newcommand{\A}{\mathbf A}
\newcommand{\C}{\mathbf C}
\newcommand{\M}{\mathbf M}
\newcommand{\T}{\mathbf T}

\title{Switching Checkerboards}

\author[1]{David Ellison}
\author[2]{Bertrand Jouve}
\author[1]{Lewi Stone}

\affil[1]{RMIT University, School of Science, Melbourne, Australia}
\affil[2]{LISST UMR5193, Toulouse University, CNRS, France}

\date{}

\begin{document}
\maketitle

\begin{abstract}
   In order to study $\M(R,C)$, the set of binary matrices with fixed row and column sums $R$ and $C$, we consider sub-matrices of the form $\begin{pmatrix}
1 & 0 \\
0 & 1
\end{pmatrix}$
and $\begin{pmatrix}
0 & 1 \\
1 & 0
\end{pmatrix}$, called positive and negative checkerboard respectively. We define an oriented graph of matrices $G(R,C)$ with vertex set $\M(R,C)$ and an arc from $\A$ to $\A'$ indicates you can reach $\A'$ by switching a negative checkerboard in $\A$ to positive. We show that $G(R,C)$ is a directed acyclic graph and identify classes of matrices which constitute unique sinks and sources of $G(R,C)$. Given $\A,\A'\in\M(R,C)$, we give necessary conditions and sufficient conditions on $\M=\A'-\A$ for the existence of a directed path from $\A$ to $\A'$.

We then consider the special case of $\M(\D)$, the set of adjacency matrices of graphs with fixed degree distribution $\D$. We define $G(\D)$ accordingly by switching negative checkerboards in symmetric pairs. We show that $Z_2$, an approximation of the spectral radius $\lambda_1$ based on the second Zagreb index, is non-decreasing along arcs of $G(\D)$. Also, $\ll$ reaches its maximum in $\M(\D)$ at a sink of $G(\D)$. We provide simulation results showing that applying successive positive switches to an Erd\H os-R\'enyi graph can significantly increase $\ll$.
\end{abstract}

\begin{keywords}
Graph theory, binary matrices, spectral radius, graph modifications
\end{keywords}

\section*{Introduction}

In the study of diffusion processes along a network, the largest eigenvalue  of the adjacency matrix, called the spectral radius of the network and denoted by $\ll$, plays a critical role. For instance, in epidemic models, having a reproduction number $R_0$ greater or smaller than $1/\ll$ determines whether the epidemic will die out or keep spreading. Also, in the problem of synchronisation of coupled oscillators on a network, a key threshold criterion for the stability of the synchronised solution is based on $\ll$. The largest eigenvalue similarly affects other percolation processes such as wildfires or rumors spreading in social networks and is also of interest in the study of random matrices.

The parameter which has the greatest impact on the value of $\ll$ is the degree distribution of the network. Indeed, some lower bounds and approximations of $\ll$ are entirely determined by the degree distribution \cite{restrepo2007approximating}. Such approximations can be very precise for the vast majority of cases, but tend to become wildly inaccurate in extreme cases of networks with very particular topological properties. In this article, we consider the degree distribution fixed in order to analyse how the topology of the network affects $\ll$. The study of matrices with fixed degree distribution has applications to epidemic models \cite{coupechoux2014clustering}, synchronisation problems \cite{mcgraw2005clustering},  ecology \cite{brualdi1999nested}, LT codes \cite{jiang2016degree} and many other areas. Once the degree distribution is fixed, the degree assortativity is the most important parameter of real networks \cite{xulvi2004reshuffling,newman2003}. 

When working with a fixed degree distribution, we rely on a basic operation which consists in switching the $1$s and $0$s in a $2\times 2$ sub-matrix of the adjacency matrix of the form $\begin{pmatrix}
1 & 0 \\
0 & 1
\end{pmatrix}$
or $\begin{pmatrix}
0 & 1 \\
1 & 0
\end{pmatrix}$, called a checkerboard. Any such switching amounts to rewiring two edges in a way that leaves the degree distribution invariant. This allows us to visualise the set of binary matrices with a given degree distribution as the vertex set of a graph where two matrices are adjacent if they differ by one switch. It was shown in \cite{ryser1957} that this graph of matrices is connected.

In this paper, we call the two above sub-matrices positive and negative checkerboards. This polarisation of checkerboards defines an orientation of the edges of the graph of matrices. In \Cref{sec:Pos-Neg}, we introduce the basic definitions surrounding checkerboards. Then, in \Cref{sec:OGM}, we show that the oriented graph of matrices is a directed acyclic graph and discuss its sources and sinks as well as when $\A'$ can be reached from $\A$ by following an oriented path. Finally, in \Cref{sec:spectral radius}, we describe how the second Zagreb index, an approximation of $\ll$, evolves throughout the graph of matrices. We then show that $\ll$ reaches its global maximum at a sink of the graph of matrices and describe the evolution of $\ll$ through simulations. The proofs of two theorems from \Cref{sec:OGM} have been placed in the appendix for the sake of improving readability. In order to maximise the generality of our results, we have considered non-symmetric binary matrices with fixed row and column sums in all sections except \Cref{sec:spectral radius}, where we consider only adjacency matrices. Note that a non-symmetric matrix $\A$ can be interpreted as a characteristic sub-matrix of $\mathbf M=\begin{pmatrix}
0&\A\\
^t\A&0
\end{pmatrix}$, where $\mathbf M$ is the adjacency matrix of a bipartite graph.

\section{Positive and Negative Switches}\label{sec:Pos-Neg}

Given a $(0,  1)$-matrix $\A$, a \emph{checkerboard} of $\A$ is a $2\times 2$ sub-matrix of $\A$ of the form either $\begin{pmatrix}
1 & 0 \\
0 & 1
\end{pmatrix}$
or $\begin{pmatrix}
0 & 1 \\
1 & 0
\end{pmatrix}$. A checkerboard found on rows $i$ and $j$ and columns $k$ and $l$, with $i<j$ and $k<l$, is said to have \emph{coordinates} $(i,j,k,l)$. A \emph{unitary checkerboard} is a checkerboard of area $1$ and has  coordinates $(i,i+1,k,k+1)$.  \emph{Switching} a checkerboard refers to replacing one form with the other. Note that row and column sums are invariant under this operation. The following result was shown by Ryser in 1957:

\begin{theorem}[\cite{ryser1957}]\label{ryser} Given  two  matrices  $\A$ and  $\A'$ in  the  class  of $(0,  1)$-matrices having specified  row and  column  sums, one can  pass  from  $\A$ to  $\A'$  by  a  finite  sequence  
of switches.
\end{theorem}

This property makes switching checkerboards an essential tool for studying classes of matrices with fixed row and column sums as well as classes of graphs with fixed degree distributions. Applying successive random switches to a binary matrix creates a Markov chain which visits all possible matrices with the same row and column sums. This Markov chain can then be used to sample randomly matrices with a fixed degree distribution \cite{artzy2005generating}. 

In this paper, we attribute a sign to checkerboards and switches: a checkerboard of the form $\begin{pmatrix}
1 & 0 \\
0 & 1
\end{pmatrix}$ is said to be \emph{positive} while a checkerboard of the form $\begin{pmatrix}
0 & 1 \\
1 & 0
\end{pmatrix}$ is said to be \emph{negative}. A \emph{positive switch} corresponds to replacing a negative checkerboard with a positive one while a \emph{negative switch} is the reverse. We define a \emph{switching matrix} $\C_{i,j,k,l}$, with $i<j$ and $k<l$, as having coefficients $c_{ik}=c_{jl}=1$, $c_{il}=c_{kj}=-1$ and 0 elsewhere. Thus, a positive (resp. negative) switch of coordinates $(i,j,k,l)$ corresponds to adding (resp. subtracting) $\C_{i,j,k,l}$. Note that $\C_{i,j,k,l}=\sum_{i\leq p<j,k\leq q<l} \C_{p,p+1,q,q+1}$; i.e. any switching matrix is a sum of \emph{unitary switching matrices}. Note also that the unitary switching matrices are linearly independent and form a basis of the space of zero-sum (for rows and columns) matrices.

\section{Oriented Graph of Matrices}\label{sec:OGM}

\subsection{Sources and Sinks}

Given vectors $R\in \N^p$ and $C\in \N^q$, let $\MM(R,C)$ denote the set of binary matrices with row and column sums $R$ and $C$. We assume $\MM(R,C)\neq\emptyset$. Let $G(R,C)$ denote the oriented graph with vertex set $\MM(R,C)$ and there is an arc from $\A$ to $\A'$ if $\A'$ can be obtained from $\A$ with a positive switch.

As mentioned in the introduction, the undirected graph underlying $G(R,C)$, called \emph{interchange graph} in \cite{brualdi1980matrices}, is well-studied in the literature. In \cite{artzy2005generating}, multiple algorithms which generate random networks with a fixed degree distribution are based on this graph. Our purpose here, however, is to focus on properties which derive from the orientation.

\begin{prop}\label{acyclic}
The graph $G(R,C)$ is connected and acyclic.
\end{prop}

\begin{proof}
For $\A\in\MM(R,C)$, let $I(\A)=\sum_{i,j}ija_{ij}$. Let $\A'$ be obtained from $\A$ with a positive switch of coordinates $(i,j,k,l)$ where $i<j$ and $k<l$; i.e there is an arc from $\A$ to $\A'$ in $G(R,C)$. Then, $$I(\A')-I(\A)=I(\C_{i,j,k,l})=(i-j)(k-l)>0.$$ Since $I$ increases along arcs, and therefore along directed paths, $G(R,C)$ is acyclic. 

The connectedness property is Ryser's theorem: a sequence of (positive or negative) switches leading from $\A$ to $\A'$ corresponds to an undirected path in $G(R,C)$.
\end{proof}

Recall that directed acyclic graphs have \emph{sources} and \emph{sinks} which are vertices with zero in-degree and out-degree, respectively. In $G(R,C)$, sources and sinks are, respectively, matrices with no positive checkerboards and no negative checkerboards.

A matrix $\A\in\MM(R,C)$ is said to be \emph{nested} (resp. anti-nested) if the sequence $01$ (resp. $10$) does not appear in any row or column; i.e if the 1s occur before the $0$s (resp. $0$s before $1$s) in every row and column. Nested graphs have important applications, in particular in ecology \cite{brualdi1999nested}. For our purpose, the nested case is both trivial  and theoretically important, as indicated by the following proposition and its corollary.

\begin{prop}\label{nested}
If $R$ and $C$ are non-increasing, then $\A\in\MM(R,C)$ is nested if and only if it has no checkerboards. 
\end{prop}

\begin{proof}
If $\A$ is nested, the sequence $01$ does not appear in any row or column of $\A$. Hence, there is no checkerboard in $\A$.

Conversely, if $\A$ is not nested, there is a row or column containing the sequence $01$. Say it is a row; since the column with the $0$ has degree at least equal to that with the $1$, there is another row containing $10$ in the same two columns. Hence, there is a checkerboard.
\end{proof}

\begin{corollary}\label{singleton}
The set $\MM(R,C)$ is a singleton if and only if it includes a matrix which becomes nested after reordering its rows and columns by non-increasing degree.
\end{corollary}

\begin{proof}
This result derives from \Cref{nested} by using \Cref{ryser} and the fact that the number of checkerboards in a matrix is invariant under row and column permutation.
\end{proof}

We define a \emph{zebra} as a matrix in $\MM(R,C)$ which is the sum of two matrices, one nested and one anti-nested. The name zebra refers to the three stripes formed by the $1$s and $0$s in the matrix, as shown in \Cref{ex:Z&A}. We say that a zebra is \emph{split vertically} (resp. \emph{horizontally}) if no column (resp. row) has $1$s from both the nested and anti-nested parts (see \Cref{ex:Z&A}). In other words, a split zebra can be split along a vertical (or horizontal) axis so that the left (resp. top) half is nested while the right (resp. bottom) half is anti-nested. Note that the two "halves" need not be of equal size. Also, we define an \emph{anti-zebra} as the complement of the vertical reflection of a zebra; i.e. $b_{ij}=1-a_{n-i,j}$ (see \Cref{ex:Z&A}).

\begin{example}[Zebras and Anti-zebras]\label{ex:Z&A}
The matrices $\A$ and $\A'$ are zebras. In $\A$, the nested and anti-nested parts overlap horizontally and vertically, so $\A$ is not split. Meanwhile, $\A'$ is a horizontally split zebra, where the top three rows are nested and the bottom three anti-nested. Matrices $\mathbf B$ and $\mathbf B'$ are anti-zebras which are the complement of the vertical reflection of $\A$ and $\A'$, so $\mathbf B'$ is horizontally split and $\mathbf B$ is not split.
$$\A=\begin{pmatrix}
1&1&1&1&1&0\\
1&1&1&1&0&0\\
1&1&0&0&0&1\\
1&0&0&0&0&1\\
1&0&0&0&1&1\\
0&0&1&1&1&1
\end{pmatrix}
\hspace{1cm} \text{and} \hspace{1cm} \mathbf{B}=\begin{pmatrix}
1&1&0&0&0&0\\
0&1&1&1&0&0\\
0&1&1&1&1&0\\
0&0&1&1&1&0\\
0&0&0&0&1&1\\
0&0&0&0&0&1
\end{pmatrix}$$
$$\A'=\begin{pmatrix}
1&1&1&1&1&0\\
1&1&1&1&0&0\\
1&1&0&0&0&0\\
0&0&0&0&0&1\\
0&0&0&0&1&1\\
0&0&1&1&1&1
\end{pmatrix}
\hspace{1cm} \text{and} \hspace{1cm} \mathbf{B}'=\begin{pmatrix}
1&1&0&0&0&0\\
1&1&1&1&0&0\\
1&1&1&1&1&0\\
0&0&1&1&1&1\\
0&0&0&0&1&1\\
0&0&0&0&0&1
\end{pmatrix}$$
\end{example}

Like nested graphs, zebras and anti-zebras can be characterised by the absence of certain sub-matrices:

\begin{claim}\label{geometry}
If we exclude full (resp. empty) rows and columns, zebras (resp. anti-zebras) are exactly the matrices without $\begin{pmatrix}
0&1\\
*&0
\end{pmatrix}$ or
$\begin{pmatrix}
0&*\\
1&0
\end{pmatrix}$ (resp.  $\begin{pmatrix}
*&1\\
1&0
\end{pmatrix}$ or
$\begin{pmatrix}
0&1\\
1&*
\end{pmatrix}$) sub-matrices, where $*$ denotes either $0$ or $1$.
\end{claim}

\begin{proof}
    The absence of those sub-matrices derives from the geometry of the zebra. The converse, which is not needed to prove the theorem, we leave as an exercise.
\end{proof}

Note that the forbidden sub-matrices include negative checkerboards; thus zebras and anti-zebras form sinks of $G(R,C)$.

\begin{claim}\label{no sub}
Split zebras, split anti-zebras and their complements include at most one of the following vectors as sub-matrix: $\begin{pmatrix}
1&0&1
\end{pmatrix}$, $\begin{pmatrix}
0&1&0
\end{pmatrix}$,
$\begin{pmatrix}
1\\
0\\
1
\end{pmatrix}$ and $\begin{pmatrix}
0\\
1\\
0
\end{pmatrix}$. In particular, horizontally split zebras have no $\begin{pmatrix}
1&0&1
\end{pmatrix}$, $\begin{pmatrix}
0&1&0
\end{pmatrix}$ and $\begin{pmatrix}
0\\
1\\
0
\end{pmatrix}$ sub-matrix. 
\end{claim}

\begin{proof}
This derives from the geometry of the split zebra and anti-zebra.
\end{proof}

\begin{theorem}\label{th:zebra}
If $\MM(R,C)$ contains a split zebra or a split anti-zebra, then that split zebra or anti-zebra is the only element of $\MM(R,C)$ without negative checkerboards; i.e it is the unique sink in $G(R,C)$.
\end{theorem}

\noindent (Proof in the appendix)

\begin{corollary}
Let $\mathbf{B}\in\MM(R,C)$ be a split zebra or a split anti-zebra. For all $\A\in \MM(R,C)$, $\mathbf{B}$ can be reached from $\A$ via a sequence of positive switches.
\end{corollary}

\begin{corollary}
If $\MM(R,C)$ contains the complement of a split zebra or a split anti-zebra, then it is the unique source in $G(R,C)$.
\end{corollary}

\begin{proof}
The complement of a unique sink is a unique source.
\end{proof}

\subsection{Adapting Ryser's Theorem to the Oriented Graph of Matrices}

It follows from \Cref{ryser} that the undirected version of $G(R,C)$ is connected. Having introduced an orientation naturally raises the question of when can a matrix $\A'$ be reached from matrix $\A$ via a sequence of positive switches. The remainder of this section is devoted to answering this question. We will denote by $\A\rar \A'$ that there is a directed path connecting $\A$ to $\A'$ in $G(R,C)$.

\begin{prop}\label{NC}
Let $R\in \N^p$, $C\in \N^q$, let $\A,\A' \in \MM(R,C)$ and let $\M=\A'-\A$. If $\A\rar \A'$, then $\M$ is a sum of unitary switching matrices; i.e. \begin{enumerate}[(i)] 
\item $\exists !\T\in \MM_{p-1,q-1}(\N)$: $\M=\sum_{i,k} t_{ik}\C_{i,i+1,k,k+1}$.
\end{enumerate}
\end{prop}

\begin{proof}
Each positive switch corresponds to adding a switching matrix and each switching matrix is a sum of unitary switching matrices. The unicity of $\T$ follows from the linear independence of the unitary switching matrices.
\end{proof}

\begin{remark}
~

\begin{itemize}
    \item Note that $\T$ is uniquely determined by $\M$, even when multiple switching sequences lead from $\A$ to $\A'$ (see \Cref{ex1}).
    \item Given $\M=\A'-\A$ which satisfies condition $(i)$, if we switch a negative checkerboard in $\A$ to positive, then the coefficients of $\T$ inside a rectangle corresponding to the checkerboard are all decreased by $1$. The new matrix $\M$ still satisfies condition $(i)$ if and only if the coefficients of $\T$ all remain non-negative. Reaching $\A'$ via successive switches corresponds to reducing $\T$ to $0$ in this fashion. (See examples below for more details.) 
\end{itemize}
\end{remark}

We denote by $(i)$ the necessary condition given in \Cref{NC}. Unfortunately, condition $(i)$ is not sufficient to ensure $\A\rar\A'$, as shown in \Cref{ex2}.

\begin{example}\label{ex1}
Let $\A=\begin{pmatrix}
0&0&1\\
1&0&0\\
1&1&0
\end{pmatrix}$ and $\A'=\begin{pmatrix}
1&0&0\\
0&1&0\\
1&0&1
\end{pmatrix}$. There are two directed paths from $\A$ to $\A'$: $\A\rar\begin{pmatrix}
1&0&0\\
0&0&1\\
1&1&0
\end{pmatrix}\rar\A'$ and $\A\rar\begin{pmatrix}
0&1&0\\
1&0&0\\
1&0&1
\end{pmatrix}\rar\A'$. In the first case, we add the switching matrices $\C_{1,2,1,3}=\begin{pmatrix}
1&0&-1\\
-1&0&1\\
0&0&0
\end{pmatrix}$ followed by $\C_{2,3,2,3}=\begin{pmatrix}
0&0&0\\
0&1&-1\\
0&-1&1
\end{pmatrix}$. In the second case, we add $\C_{1,3,2,3}=\begin{pmatrix}
0&1&-1\\
0&0&0\\
0&-1&1
\end{pmatrix}$ and $\C_{1,2,1,2}=\begin{pmatrix}
1&-1&0\\
-1&1&0\\
0&0&0
\end{pmatrix}$. We have $\M=\A'-\A=\begin{pmatrix}
1&0&-1\\
-1&1&0\\
0&-1&1
\end{pmatrix}=\C_{1,2,1,3}+\C_{2,3,2,3}=\C_{1,3,2,3}+\C_{1,2,1,2}$. In both of these sums, the first switching matrix can be split into two unitary switching matrices. So $\M$ is the sum of three unitary switching matrices: $\M=\C_{1,2,1,2}+\C_{1,2,2,3}+\C_{2,3,2,3}$. Thus $\M$ satisfies condition $(i)$ with $\T=\begin{pmatrix}
1&1\\
0&1
\end{pmatrix}$. 
\end{example}

\begin{example}\label{ex2}
Let $\A=\begin{pmatrix}
0 & 0 & 0 & 1 \\
1 & 1 & 0 & 1 \\
1 & 0 & 1 & 1 \\
1 & 0 & 0 & 0
\end{pmatrix}$ and $\A'=\begin{pmatrix}
1 & 0 & 0 & 0 \\
1 & 0 & 1 & 1 \\
1 & 1 & 0 & 1 \\
0 & 0 & 0 & 1
\end{pmatrix}$.
We have $\M=\A'-\A=\begin{pmatrix}
1 & 0 & 0 & -1 \\
0 & -1 & 1 & 0 \\
0 & 1 & -1 & 0 \\
-1 & 0 & 0 & 1
\end{pmatrix}$, which satisfies condition $(i)$ with \\
$\T=\begin{pmatrix}
1 & 1 & 1  \\
1 & 0 & 1  \\
1 & 1 & 1
\end{pmatrix}$. Matrix $\A$ has only one negative checkerboard, with coordinates $(1,4,1,4)$. After switching it to positive, reaching $\A'$ now requires a negative switch of coordinates $(2,3,2,3)$. Indeed, switching that checkerboard decrements all coefficients of $T$, leaving a -1 in the centre. Hence $\A\not\rar \A'$. While the matrix $\M$ is the sum of eight unitary switching matrices, and can be written as a sum of (not all unitary) switching matrices in many ways, none of these sums includes $\C_{1,4,1,4}$. So none of the switches appearing in these sums is feasible in $\A$. 
\end{example}

\begin{lemma}\label{m_ij}
Let $\A,\A' \in \MM(R,C)$ and $\M=\A'-\A$ such that there exists $\T\in \MM_{p-1,q-1}(\N)$ such that $\M=\sum_{i,k} t_{ik}\C_{i,i+1,k,k+1}$ ~$(i)$. Let us extend $T$ so that $t_{ij}=0$ if $i=0$ or $p$, or if $j=0$ or $q$. Then, for all $i,j$, we have:  $$m_{ij}=t_{ij}+t_{i-1,j-1}-t_{i,j-1}-t_{i-1,j}.$$ 
\end{lemma}

\begin{proof}
Each switching matrix has four non-zero coefficients. So in condition $(i)$, exactly four terms in the sum contribute to $m_{ij}$.
\end{proof}

\begin{figure}[ht]
    \begin{center}
    \begin{tikzpicture}[scale=1.5]
        \draw [gray] (0,0) grid (2,2); 
         
        \node [above] at (0,.80) {$m_{i,j-1}$};
        \node [above] at (1,.80) {$m_{i,j}$};
        \node [above] at (2,.80) {$m_{i,j+1}$};
        \node [above] at (0,1.80) {$m_{i-1,j-1}$};
        \node [above] at (1,1.80) {$m_{i-1,j}$};
        \node [above] at (2,1.80) {$m_{i-1,j+1}$};
        \node [above] at (0,-.20) {$m_{i+1,j-1}$};
        \node [above] at (1,-.20) {$m_{i+1,j}$};
        \node [above] at (2,-.20) {$m_{i+1,j+1}$};
        
        \node [above] at (0.5,0.3) {$t_{i,j-1}$};
        \node [above] at (0.5,1.3) {$t_{i-1,j-1}$};
        \node [above] at (1.5,0.3) {$t_{i,j}$};
        \node [above] at (1.5,1.3) {$t_{i-1,j}$};
        
    \end{tikzpicture}   
    
    \medskip
    
    $m_{ij}=t_{ij}+t_{i-1,j-1}-t_{i,j-1}-t_{i-1,j}$
    \end{center}
    \caption{Relation between coefficients of $\M$ and $\T$}
    \label{fig.mij}
\end{figure}

If $\M$ satisfies $(i)$, we create a $(p-1)\times(q-1)$ grid, with the coefficients of $\T$ inside the cells and the coefficients of $\M$ at the corners (see \Cref{fig.mij}). Each cell corresponds to a unitary switching matrix and each coefficient of $\T$ indicates how many times its cell needs to be switched in order to go from $\A$ to $\A'$. Let $m=\max t_{ik}$. For $1\leq i\leq m$, we define $\P_i(\M)$ as the reunion of the cells of the grid with coefficients at least $i$ (see \Cref{ex3}). Note that $\P_i(\M)$ is formed by one or several polyominoes. We recall that a \emph{polyomino} is a shape formed by a finite number of orthogonally connected cells in a square grid.

\begin{example}\label{ex3}
Let $\A=\begin{pmatrix}
0 & 0 & 1 & 1 \\
0 & 0 & 1 & 1 \\
1 & 1 & 0 & 0 \\
1 & 1 & 0 & 0
\end{pmatrix}$ and $\A'=\begin{pmatrix}
1 & 1 & 0 & 0 \\
1 & 1 & 0 & 0\\
0 & 0 & 1 & 1 \\
0 & 0 & 1 & 1 
\end{pmatrix}$.
We have $\M=\A'-\A=\begin{pmatrix}
1 & 1 & -1 & -1 \\
1 & 1 & -1 & -1 \\
-1 & -1 & 1 & 1 \\
-1 & -1 & 1 & 1
\end{pmatrix}$, which satisfies condition $(i)$ with \\ $\T=\begin{pmatrix}
1 & 2 & 1  \\
2 & 4 & 2  \\
1 & 2 & 1
\end{pmatrix}$. Thus, $\P_1(\M)$ is a $3\times 3$ square, $\P_2(\M)$ is an X-shaped pentomino (see \Cref{fig.polyomino}) and $\P_3(\M)$ and $\P_4(\M)$ are monominoes surrounding only the central cell. Note that we have $\A\rar\A'$ and going from $\A$ to $\A'$ requires at minimum four switches.

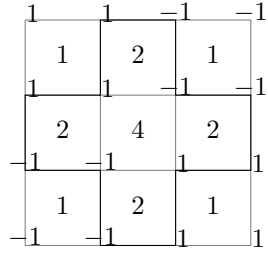
\begin{figure}[ht]
    \begin{center}
    \begin{tikzpicture}[scale=1]
        \draw [gray] (0,0) grid (3,3);
        \draw (1,0)--(2,0)--(2,1)--(3,1)--(3,2)--(2,2)--(2,3)--(1,3)--(1,2)--(0,2)--(0,1)--(1,1)--(1,0);
         
        \node [above] at (0,.85) {$-1$};
        \node [above] at (1,.85) {$-1$};
        \node [above] at (2.1,.85) {$1$};
        \node [above] at (3.1,.85) {$1$};
        \node [above] at (0.11,1.85) {$1$};
        \node [above] at (1.1,1.85) {$1$};
        \node [above] at (2,1.85) {$-1$};
        \node [above] at (3,1.85) {$-1$};
        \node [above] at (0,-.15) {$-1$};
        \node [above] at (1,-.15) {$-1$};
        \node [above] at (2.1,-.15) {$1$};
        \node [above] at (3.1,-.15) {$1$};
        \node [above] at (0.1,2.85) {$1$};
        \node [above] at (1.1,2.85) {$1$};
        \node [above] at (2,2.85) {$-1$};
        \node [above] at (3,2.85) {$-1$};
        
        \node [above] at (0.5,0.3) {$1$};
        \node [above] at (0.5,1.3) {$2$};
        \node [above] at (1.5,0.3) {$2$};
        \node [above] at (1.5,1.3) {$4$};
        \node [above] at (2.5,0.3) {$1$};
        \node [above] at (2.5,1.3) {$2$};
        \node [above] at (0.5,2.3) {$1$};
        \node [above] at (1.5,2.3) {$2$};
        \node [above] at (2.5,2.3) {$1$};
    \end{tikzpicture}   
    
    \end{center}
    \caption{Polyomino $\P_2(\M)$ of \Cref{ex3}. Coefficients of $\M$ are placed at the intersections of the grid and coefficients of $\T$ are inside the cells.}
    \label{fig.polyomino}
\end{figure}

\end{example}

\begin{lemma}\label{no diag touching}
If $\P_i(\M)$ contains two diagonally adjacent cells, then at least one of the two common neighbouring cells is also in $\P_i(\M)$.
\end{lemma}

\begin{proof}
Say the cells of coordinates $(i,j)$ and $(i-1,j-1)$ are in $\P_i(\M)$ and the cells of coordinates $(i-1,j)$ and $(i,j-1)$ are not in $\P_i(\M)$ (or vice versa). Then, $t_{ij},t_{i-1,j-1}\geq i$ (resp. $\leq i-1$) and $t_{i-1,j},t_{i,j-1}\leq i-1$ (resp. $\geq i$). It follows from \Cref{m_ij} that $m_{ij}\geq 2$ (resp. $\leq -2$). Since $\M=\A'-\A$ and $\A$ and $\A'$ have binary coefficients, the coefficients of $\M$ are in $\{-1,0,1\}$. This is impossible.
\end{proof}

\begin{remark}\label{touching polyo}
This means that $\P_i(\M)$ cannot contain two polyominoes connected by a corner.
\end{remark}

\begin{theorem}\label{SC}
Let $R\in \N^p$, $C\in \N^q$, let $\A,\A' \in \MM(R,C)$ and let $\M=\A'-\A$. If $\M$ satisfies the following conditions:
\begin{enumerate}[(i)]
    \item $\exists \T\in \MM_{p-1,q-1}(\N)$: $\M=\sum_{i,k} t_{ik}\C_{i,i+1,k,k+1}$,
    \item For all $i$, each connected component of $\P_i(\M)$ is simply connected,
    \item $|i'-i|\leq 1$ and $|k'-k|\leq 1 \imp |t_{i'k'}-t_{ik}|\leq 1$,
\end{enumerate}
then $\A\rar \A'$.
\end{theorem}

\noindent (Proof in the appendix)
\bigskip

Note that the third condition says that the coefficients inside orthogonally or diagonally adjacent cells must be equal or successive integers. Condition $(iii)$ in \Cref{SC} is probably unnecessary and it seems that if for some $i$, $\P_i(\M)$ is not simply connected, it should be possible to create an instance where $\A\not\rar\A'$, like in \Cref{ex2}. Combining these two observations yields the following conjecture:

\begin{conj}
Let $R\in \N^p$, $C\in \N^q$ and let $\M\in \MM_{p,q}(\{-1,0,1\})$. We have $\A\rar \A'$ for all $\A,\A'\in\MM(R,C)$ such that $\A'-\A=\M$ if and only if $\M$ satisfies:
\begin{enumerate}[(i)]
    \item $\exists \T\in \MM_{p-1,q-1}(\N)$: $\M=\sum_{i,k} t_{ik}\C_{i,i+1,k,k+1}$,
    \item For all $i$, each connected component of $\P_i(\M)$ is simply connected.
\end{enumerate}
\end{conj}

\section{Spectral Radius and Second Zagreb index}\label{sec:spectral radius}

\subsection{Extrema}

In this section, we consider the class of adjacency matrices of simple graphs with fixed degree distribution $\D$. We investigate how the topology of simple graphs affects their spectral radius. To this purpose, we analyse the effect of successive positive checkerboard switches on the spectral radius $\ll$ as well as on the second Zagreb index $M_2$, which was used in \cite{abdo2014estimating} to create an approximation of $\ll$: $Z_2=\sqrt\frac{M_2}m$.
Amongst all the existing approximations for $\ll$, we chose to focus on $Z_2$ because we are able to map very precisely how it varies throughout the set of adjacency matrices of fixed degree distribution. In contrast, most other commonly used approximations are determined by the degree distribution, meaning that they are invariant under checkerboard switching. \cite{pastor2018eigenvector,restrepo2007approximating}. 

Let $G=(V,E)$ be a simple graph with $|V|=n$ and $|E|=m$ and adjacency matrix $\A$. Let $d_i$ denote the degree of vertex $i\in V$. The \emph{spectral radius} of $G$, denoted by $\ll(G)$ or $\ll(\A)$, is the largest eigenvalue of $\A$. The \emph{first} and \emph{second Zagreb indices} are defined as $M_1=\sum_{i\in V} d_i^2$ and $M_2=\sum_{ij\in E}d_id_j$ and we have $Z_1=\sqrt\frac{M_1}n$ and $Z_2=\sqrt\frac{M_2}m$. Note that $Z_1$ is the quadratic average over $V$ of the degrees, while $Z_2$ is the quadratic average over $E$ of $\sqrt{d_id_j}$. It is well-known that $\ll\geq Z_1$, with equality if $G$ is a regular graph \cite{elphick2015relations}. Thus, it is the heterogeneity of $\D$ that allows $\ll$ and $Z_2$ to vary among the class of graphs with degree distribution $\D$. Also, for a fixed degree distribution, $M_2$ is proportional to the \emph{degree assortativity coefficient} $r$,  which is the standard Pearson coefficient for correlation between the degrees \cite{Li2005supplemental}:
$$r=\frac{M_2-(\sum_{i=1}^n\frac 12d_i^2)^2/m}{\sum_{i=1}^n\frac 12 d_i^3-(\sum_{i=1}^n\frac 12d_i^2)^2/m}.$$

Let $\D\in \N^n$ be the degree distribution of a simple graph. We now consider the class $\MM(\D)$ of adjacency matrices of simple graphs of order $n$ with fixed degree distribution $\D$; i.e. symmetric binary matrices with zeroes on the diagonal and row and column sums $\D$. Due to the absence of loops, we only consider, in this section, checkerboards with no coefficient on the diagonal. Also, given the symmetry of the matrices, checkerboards always come in symmetric pairs which are always switched together. In terms of graphs, a checkerboard corresponds to a 4-cycle of alternating edges and non-edges. Switching a checkerboard means switching the edges and non-edges. We define the \emph{symmetric switching matrix} $\CC_{ijkl}=\C_{ijkl}\;+\;^t\C_{ijkl}$. Note that we now require the coordinates $(i,j,k,l)$ to be all different. The distinction between positive and negative checkerboards for graphs is dependent on an ordering of the vertices. We will always sort the vertices by non-increasing degree; i.e $\D$ is non-increasing. A positive (resp. negative) switch of coordinates $(i,j,k,l)$ now corresponds to adding (resp. subtracting) $\CC_{ijkl}$ to the adjacency matrix. We define $G(\D)$ as the directed graph with vertex set $\MM(\D)$ and an arc joins $\A$ to $\A'$ when $\A'$ can be obtained from $\A$ by a positive switch.

\begin{prop}
The graph $G(\D)$ is connected and acyclic.
\end{prop}

\begin{proof}
The connectedness is shown in \cite{berge1970graphes}. The acyclicity follows from the acyclicity in the asymmetric case. 
\end{proof}

\begin{lemma}\label{M2 switch}
Let $\A,\A'\in\MM(\D)$ be such that $\A'$ can be obtained from $\A$ by a positive switch of coordinates $(i,j,k,l)$. We have $M_2(\A')-M_2(\A)=(d_i-d_j)(d_k-d_l)$.
\end{lemma}

\begin{proof}
$\A'$ is obtained from $\A$ by adding the edges $ik$ and $jl$ and removing $il$ and $kj$. So $M_2(\A')= M_2(\A)+d_id_k+d_jd_l-d_id_l-d_jd_k$.
\end{proof}

\begin{prop}
$M_2$ and $Z_2$ are non-decreasing along arcs of $G(\D)$.
\end{prop}

\begin{proof}
This follows from \Cref{M2 switch} and $\D$ being non-decreasing.
\end{proof}

Note that $M_2$ and $Z_2$ are constant only along arcs of $G(\D)$ which correspond to a checkerboard of coordinates $(i,j,k,l)$ where $d_i=d_j$ or $d_k=d_l$. 

\begin{corollary}
The sources and sinks of $G(\D)$ are, respectively, the local minima and maxima of $M_2$ and $Z_2$. 
\end{corollary}

The above results give a clear picture of the variations of $M_2$ and $Z_2$ along $G(\D)$. Unfortunately, the variations of $\ll$ are not quite so neatly organised. The following lemma gives a lower bound to the effect of a single positive switch on the spectral radius:

\begin{lemma}\label{lambda1 switch}
Let $\A,\A'\in\MM(\D)$ such that $\A'$ can be obtained from $\A$ by a sequence of positive switches of coordinates $(i_p,j_p,k_p,l_p)$. Let $X$ denote the normalised principal eigenvector of $\A$. We have $$\ll(\A')-\ll(\A)\geq 2\sum_p (x_{i_p}-x_{j_p})(x_{k_p}-x_{l_p}).$$
\end{lemma}

\begin{proof}
Let $X'$ denote the normalised principal eigenvector of $\A'$. Recall that $$\ll(\A)=~^tX\A X=\max_{x\in \R^n,||x||=1}~^tx\A x,$$
where $||~||$ is the $L^2$-norm. We have:
$$\ll(\A')-\ll(\A)=~^tX'\A'X'-~^tX\A X \geq  ~^tX\A'X-~^tX\A X$$
$$=~^tX\sum_p\CC_{i_p,j_p,k_p,l_p}X=2\sum_p (x_{i_p}-x_{j_p})(x_{k_p}-x_{l_p}).$$
\end{proof}

Unfortunately, the coefficients of the principal eigenvector are not always in the same order as the degrees. They are, however, strongly correlated \cite{li2015correlation}; so $\ll$ increases along most arcs of $G(\D)$. In fact, they are perfectly correlated in the case where $\ll$ is maximum:

\begin{prop}\label{eigenvector order}
If $\A$ realises the maximum of $\ll$ over $\MM(\D)$, then the coefficients of the principal eigenvector $X$ of $\A$ are in the same order as the degrees; i.e. $d_i> d_j\imp x_i \geq x_j$.
\end{prop}

\begin{proof}
Assume we have $d_i> d_j$ and $x_i<x_j$. Let $E$ be a set of $d_i-d_j$ vertices adjacent to $i$ and not $j$, with $j\notin E$. Let $\A'$ be obtained from $\A$ by replacing the edges between $i$ and $E$ with edges between $j$ and $E$. Note that all the non-zero coefficients of $\A'-\A$ are in rows and columns $i$ and $j$. Let $X'$ be the normalised principal eigenvector of $\A'$. We have:
$$\ll(\A')-\ll(\A)\geq  ~^tX\A'X-~^tX\A X=2(x_j-x_i)\sum_{k\in E}x_k\geq 0. $$

Consider the case where $2(x_j-x_i)\sum_{k\in E}x_k= 0$. According to the Perron-Froebenius theorem for non-negative matrices, $X$ has non-negative coefficients. Thus, since $x_j>x_i$, $\forall k\in E, x_k=0$. Let $k\in E$, then $\sum_l a_{kl}x_l=\ll(\A)x_k=0$. Since all terms in $\sum_l a_{kl}x_l$ are non-negative and $k$ is adjacent to $i$, we deduce that $x_i=0$. Similarly, any vertex $l$ in the connected component of $i$ has $x_l=0$. Since $x_j>0$, $\ll(\A)$ is the spectral radius of the connected component of $j$, to which $i$ does not belong. Thus, the connected component of $j$ in $\A$ is a proper sub-graph of that in $\A'$. So $\ll(\A')>\ll(\A)$ in this case and in all cases. 

Let $\A''$ be obtained from $\A'$ by switching the $i$-th and $j$-th rows and the $i$-th and $j$-th columns. $\A''$ has the same row and column sums as $\A$ and we have $\ll(\A'')=\ll(\A')>\ll(\A)$. Thus, $\A$ does not maximise $\ll$.
\end{proof}

From this, we deduce the following:

\begin{theorem}\label{max radius}
The maximum of $\ll$ is reached at a sink of $G(\D)$.
\end{theorem}

\begin{proof}
Let $\A$ maximise $\ll$ over $\MM(\D)$. For each set of same-degree vertices, we can reorder the rows and columns of $\A$ so that within these sets, the coefficients of the principal eigenvector are non-increasing. This operation does not affect $\ll$. It follows from \Cref{eigenvector order} that all the coefficients of $X$ are now non-increasing. Let $\A'$ be obtained from $\A$ by a sequence of positive switches of coordinates $(i_p,j_p,k_p,l_p)$. It follows from \Cref{lambda1 switch} that $$\ll(\A')-\ll(\A)\geq 2\sum_p(x_{i_p}-x_{j_p})(x_{k_p}-x_{l_p})\geq 0.$$ Since $\ll(\A)$ is maximum, any matrix $\A'\in\MM(\D)$ such that $\A\rar\A'$ also maximises $\ll$, including the sinks that can be reached from $\A$.
\end{proof}

\subsection{Switching Algorithm Simulations}

While \Cref{max radius} tells us that at least one sink of $G(\D)$ realises the global maximum of $\ll$, it gives no indications for identifying which sink to aim for when $G(\D)$ contains several, as is usually the case. In this section, we showcase how applying successive random switches to Erd\H os-Rényi and Small-World graphs, after ordering its vertices by degree, leads to an increase in $\ll$, which is significant when the edge density is low. 

\begin{figure}[H]
    \centering
    \includegraphics[width=3.9cm]{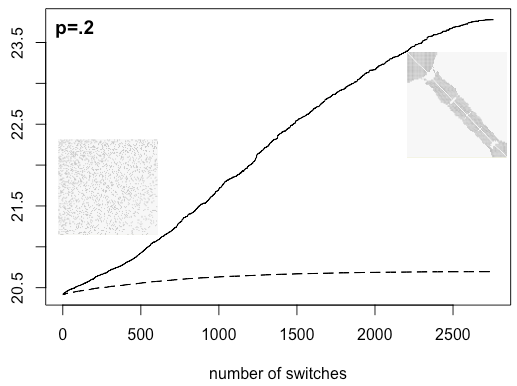}
    \includegraphics[width=3.8cm]{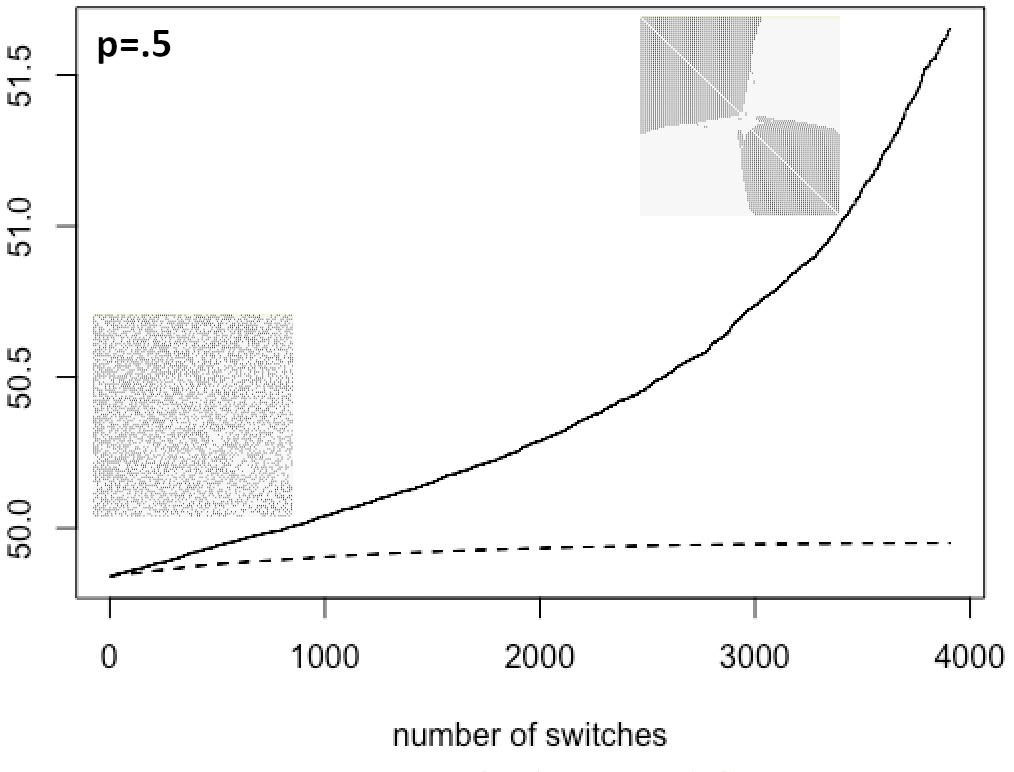}
    \includegraphics[width=4cm]{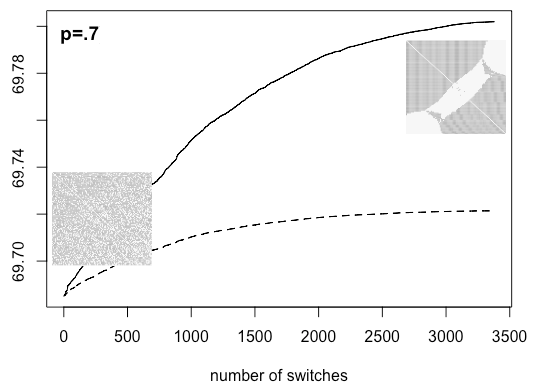}
    \caption{Spectral radius (plain) and $Z_2$ (dotted) as functions of the number of positive switches, starting from Erd\H os-R\'enyi graphs with $N=100$ vertices. From left to right, the edge density $p$ ranges from $0.2$ to $0.7$. Both the initial (Erd\H os-R\'enyi with vertices ordered by degree) and the final matrices are shown in each case. }
    \label{fig:ER_graphs}
\end{figure}

Instead of using a Monte-Carlo method as in the Xulvi-Brunet Sokolov algorithm \cite{xulvi2004reshuffling}, we only apply positive switches selected randomly in the adjacency matrix. Simulations that begin with an Erd\H os-R\'enyi graph $E(N,p)$ with vertices ordered by degree are shown in \Cref{fig:ER_graphs} and Small World graphs appear in \Cref{fig:SW_graphs}. In both cases, the end result of the switching process is a zebra with some noise when $p>0.5$ and an anti-zebra with some noise when $p<0.5$. When $p=0.5$, the end point is simultaneously almost a zebra and almost an anti-zebra. In the Erd\H os-R\'enyi case with $p=0.2$, the spectral radius increases by more than $15\%$, which is very impressive for a fixed degree distribution. Conversely, when $p=0.7$, the spectral radius increases only by a few decimal points.

\begin{figure}[H]
    \centering
    \includegraphics[height=4cm]{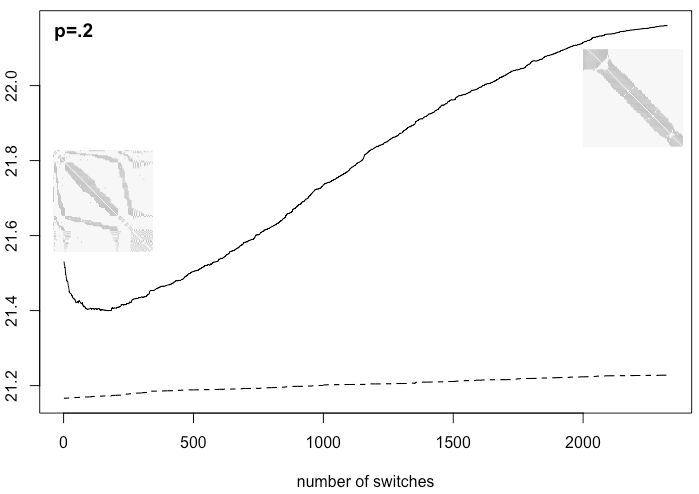}
    \includegraphics[height=4cm]{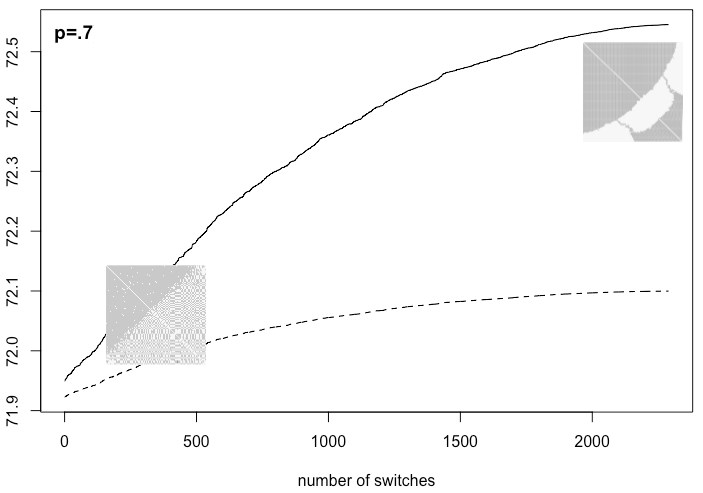}
    \caption{Same as \Cref{fig:ER_graphs} with Small-World graphs obtained from a $10 \times 10$ regular grid after a random rewiring of 10\% of the edges. }
    \label{fig:SW_graphs}
\end{figure}

\section{Conclusion}

Our analysis of the oriented graph of matrices has given us the tools needed to study how the topology of graphs impacts their spectral radius $\ll$. We have shown that the global maximum of $\ll$ is reached for a matrix without negative checkerboards. Our simulations show that successively switching negative checkerboards to positive can yield a high increase in $\ll$, especially for sparse graphs.

\section*{Acknowledgements} We would like to thank Cristophe Crespelle for a fruitful research session in Bergen, where \Cref{ex2} was discovered and \Cref{SC} was intuited, and the LabEx SMS ANR-11-LABEX-0066 for providing mobility funding.

\section*{Appendix: Proofs of Theorems 2 and 3}\label{sec:proofs}

\subsection{Proof of \Cref{th:zebra}}

\begin{theorem*}
If $\MM(R,C)$ contains a split zebra or a split anti-zebra, then that split zebra or anti-zebra is the only element of $\MM(R,C)$ without negative checkerboards; i.e it is the unique sink in $G(R,C)$.
\end{theorem*}

\begin{proof}[Proof of theorem 2]
Let $\A\in\MM(R,C)$ be a zebra with a horizontal split (i.e. the top half is nested and the bottom anti-nested) and let $\A'\in\MM(R,C)$, $\A'\neq \A$. We will show that $\A'$ has a negative checkerboard. Let $\M=\A'-\A$. Since $\A$ and $\A'$ have the same row and column sums, $\M$ has as many 1s and -1s in each row and column. We can thus choose in each row and column of $\M$ a matching associating each 1 to a -1. Let us now choose a -1 in $\M$ as a starting point for the sequence defined by the following rules: a -1 is followed by its paired 1 in the same row; a 1 is followed by its paired -1 in the same column. Since the matrix $\M$ is finite, the sequence must form a cycle. A 1-11 sub-sequence in $\M$ corresponds to a 010 sequence in $\A$, with the first 0 in the same column as the 1 and the second 0 in the same row. Thus, it follows from  \Cref{geometry} that at a -1, the cycle must form a right turn. Also, after turning right (resp. left) at a 1 in $\M$ (0 in $\A$), the following -1 is located in the same column on the other side (resp. the same side) of the split in $\A$.

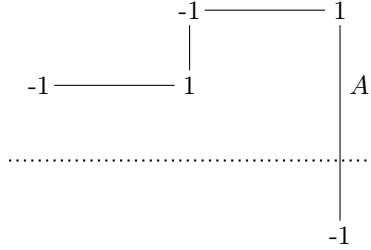
\begin{figure}[ht]
    \begin{center}
    \begin{tikzpicture}[scale=1]

        \node at (0,1) {-1};
        \node at (2,2) {-1};
        \node at (4,-1) {-1};
        \node at (4,2) {1};
        \node at (2,1) {1};

        \draw (0.2,1)--(1.8,1);
        \draw (2.2,2)--(3.8,2);
        \draw (2,1.2)--(2,1.8);
        \draw (4,-0.8)--(4,1.8);

        \draw[dotted, thick] (-0.4,0)--(4.4,0);
        
        \node [right] at (4,1) {$A$};

    \end{tikzpicture}
    \end{center}
    \caption{Section of a cycle with a left turn. The dotted line indicates the horizontal split of the zebra.}
    \label{pan}
\end{figure}

If the cycle turns left at every 1 (and right at every -1), it forms an infinite staircase pattern, which is impossible. Assume that there is at least one left turn at a 1. There must be a left turn at a 1 followed by a right turn at the next 1 somewhere in the cycle. This yields a -11-11-1 sequence with a left turn at the first 1 and a right turn at the second, as shown in \Cref{pan}. Note that the first four digits are located on the same side of the split while the final -1 is on the other side. We now consider the coefficient located at the intersection point $A$ of the row containing the first two digits and the column with the last two. Note that $A$ must be positioned between those last two digits due to the final -1 being on the other side of the split. If that coefficient is a 0 in $\A'$, then it forms a negative checkerboard in $\A'$ together with the middle three coefficients of our sequence; and thus $\A'$ is not a sink. Assume now that it is a 1 in $\A'$. The first two terms of our sequence are respectively 1 and 0 in $\A$. According to \Cref{no sub}, there is no $\begin{pmatrix}
1 & 0 & 1
\end{pmatrix}$ sub-matrix in $\A$. So the coefficient at $A$ is 0 in $\A$ and 1 in $\M$. Thus, we can shorten our cycle by replacing the middle three coefficients of our sequence with this 1. The resulting cycle has one fewer left turns. By repeating this operation, we obtain a cycle with no left turn.

Let us now consider the corresponding cycle in $\A'$. It satisfies the following properties:
\begin{itemize}
    \item It alternates between 0s and 1s, going from 0 to 1 horizontally and 1 to 0 vertically.
    \item It always turns to the right.
    \item All vertical segments cross the split.
\end{itemize}

If this cycle self-intersects, it forms one of the patterns shown in \Cref{spiralsbis} (possibly rotated by $180^{\circ}$). In case a) (resp. b)), if the coefficient located at intersection point $A$ is a $0$ (resp $1$), it forms a negative checkerboard with the second, third and fourth (resp. fourth, fifth and sixth) terms of the sequence. Else, if the coefficient located at $B$ is a $0$ (resp. $1$), $A$, $B$ and the fifth and sixth terms (resp. second and third terms) form a negative checkerboard. Else, if $B$ is a $1$ (resp. $0$), the sequence can be shortened by replacing the five middle terms with $B$. This removes the intersection at $A$ while maintaining the three properties of the cycle. By repeating this operation, we obtain a cycle with no left turn and no intersection; i.e. a negative checkerboard. Thus $\A'$ must have a negative checkerboard.

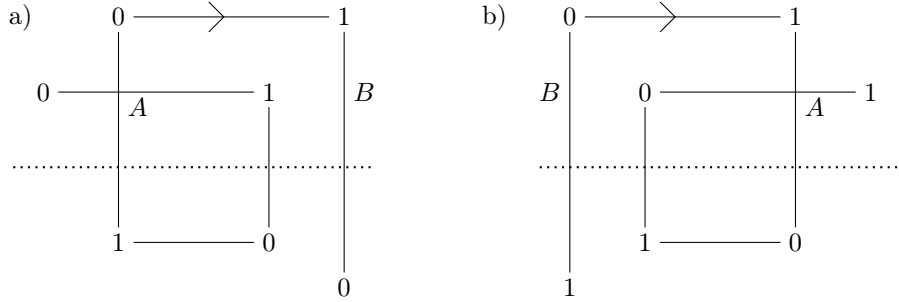
\begin{figure}[H]
    \begin{center}
    \begin{tikzpicture}[scale=1]
        
        \node at (-0.3,2) {a)};
        \node at (0,1) {0};
        \node at (1,2) {0};
        \node at (3,-1) {0};
        \node at (4,-1.6) {0};
        \node at (1,-1) {1};
        \node at (3,1) {1};
        \node at (4,2) {1};
        
        \draw (0.2,1)--(2.8,1);
        \draw (1.2,-1)--(2.8,-1);
        \draw (1.2,2)--(3.8,2);
        \draw (1,-0.8)--(1,1.8);
        \draw (3,-0.8)--(3,0.8);
        \draw (4,-1.4)--(4,1.8);
        \draw (2.2,2.2)--(2.4,2)--(2.2,1.8);
        
        \draw[dotted, thick] (-0.4,0)--(4.4,0);
        
        \node [right] at (1,0.8) {$A$};
        \node [right] at (4,1) {$B$};
        
        \node at (6,2) {b)};
        \node at (11,1) {1};
        \node at (10,2) {1};
        \node at (8,-1) {1};
        \node at (7,-1.6) {1};
        \node at (10,-1) {0};
        \node at (8,1) {0};
        \node at (7,2) {0};
        
        \draw (8.2,1)--(10.8,1);
        \draw (8.2,-1)--(9.8,-1);
        \draw (7.2,2)--(9.8,2);
        \draw (10,-0.8)--(10,1.8);
        \draw (8,-0.8)--(8,0.8);
        \draw (7,-1.4)--(7,1.8);
        \draw (8.2,2.2)--(8.4,2)--(8.2,1.8);
        
        \draw[dotted, thick] (6.6,0)--(11.4,0);
        
        \node [right] at (10,0.8) {$A$};
        \node [left] at (7,1) {$B$};
        
    \end{tikzpicture}
    \end{center}
    \caption{Section of a self-intersecting cycle with no left turn. The dotted line indicates the horizontal split of the zebra.}
    \label{spiralsbis}
\end{figure}

The case of vertically split zebras follows by symmetry along the main diagonal. The vertical reflection of a unique sink is a unique source; and the complement of a unique source is a unique sink. Thus, the same property holds if $\A$ is a split anti-zebra.

\end{proof}

\subsection{Proof of \Cref{SC}}

\begin{theorem*}
Let $R\in \N^p$, $C\in \N^q$, let $\A,\A' \in \MM(R,C)$ and let $\M=\A'-\A$. If $\M$ satisfies the following conditions:
\begin{enumerate}[(i)]
    \item $\exists \T\in \MM_{p-1,q-1}(\N)$: $\M=\sum_{i,k} t_{ik}\C_{i,i+1,k,k+1}$,
    \item For all $i$, each connected component of $\P_i(\M)$ is simply connected,
    \item $|i'-i|\leq 1$ and $|k'-k|\leq 1 \imp |t_{i'k'}-t_{ik}|\leq 1$,
\end{enumerate}
then $\A\rar \A'$.
\end{theorem*}

In order to prove \Cref{SC}, we will first prove that one of the shapes described in \Cref{pans} must appear on the contour of any simply connected polyomino. We will then show that where this shape appears on $\P_m(\M)$, we can locate a checkerboard to switch. Repeating the operation will create a directed path from $\A$ to $\A'$.

\begin{figure}[ht]
    \begin{center}
    \begin{tikzpicture}[scale=0.5]
        \node [above] at (-1,2) {1)};
        \draw (0,0) rectangle (3,2);
        \node [above] at (0,2) {$A$};
        \node [below] at (0,0) {$B$};
        \node [below] at (3,0) {$C$};
        \node [above] at (3,2) {$D$};
        
        \node [above] at (5,2) {2)};
        \draw (6,2)--(8,2)--(8,0)--(11,0)--(11,2)--(13,2);
        \node [above] at (8,2) {$A$};
        \node [below] at (8,0) {$B$};
        \node [below] at (11,0) {$C$};
        \node [above] at (11,2) {$D$};
        
        \node [above] at (15,2) {3)};
        \draw (16,2)--(18,2)--(18,0)--(21,0)--(21,3);
        \node [above] at (18,2) {$A$};
        \node [below] at (18,0) {$B$};
        \node [below] at (21,0) {$C$};
        \node [right] at (21,2) {$D$};
        
    \end{tikzpicture}    
    \end{center}
    
    \medskip
    
    1) Motif 1 occurs when the entire polyomino is a rectangle.
   
    \bigskip
    2) In motif 2, the two vertical sides must have equal length; i.e. $AB=CD$. Also, the inside of rectangle $ABCD$ must be entirely included in the polyomino. The lengths are variable and the motif may be rotated.
    
    \bigskip
    3) In motif 3, the rightmost side is longer than $AB$. The inside of rectangle $ABCD$ must be entirely included in the polyomino. The lengths are variable and the motif may be rotated or reflected.
    
    \caption{Motifs from the contour of a polyomino}
    \label{pans}
\end{figure}

\begin{lemma}\label{shadok}
On the contour of any simply connected polyomino, there appears at least one of the three motifs described in \Cref{pans}.
\end{lemma}

\begin{proof}
Let $\P$ be a simply connected polyomino. If $\P$ is a rectangle, we have motif 1. We will now assume that $\P$ is not a rectangle.

We define \emph{left} (resp. \emph{right}) corners of $\P$ as corners where the contour of $\P$, when followed clockwise, turns left (resp. right). Note that $\P$ is locally non-convex (resp. convex) around left (resp. right) corners.

For $x,y\in \P$, we define an \emph{orthogonal path from $x$ to $y$} as a line joining $x$ and $y$ consisting of only horizontal and vertical segments included inside $\P$. We denote by $OP(x,y)$ the set of orthogonal paths joining $x$ and $y$. 
We then define the \emph{length} of {an orthogonal path $\L$} as the pair $(N_\L,l_\L)$, where $N_{\L}$ is the number of segments of $\L$ and $l_{\L}$ is the length of the last segment. We define $d(x,y)$, the \emph{distance} between $x$ and $y$, as the length of the minimal orthogonal path, going by lexicographic order (ie. $d(x,y)= (N_0,l_0)$ with $N_0=\min_{\L\in OP(x,y)}N_{\L}$ and $l_0=\min_{\L\in OP(x,y); N_\L=N_0}l_{\L}$). \footnote{Note that $d$ is not a distance in the usual sense, as its co-domain is not $\R^+$. Yet, while $d':(x,y)\mapsto N_0+\frac2\pi \arctan l_0$ defines an actual distance, $d$ is more practical for our purposes.}

Since $\P$ is not a rectangle, it is not convex; so we may choose a point $x_0\in \P$ such that $\P$ is not starred in $x_0$. The set of points in $\P$ which maximise the distance to $x_0$ comprises one or several segments from the contour of $\P$. Let $B$ and $C$ be the ends of one such segment. Both $B$ and $C$ must be corners of $\P$, else points on $(B,C)$ outside of $[B,C]$ would be at an equal or greater distance from $x_0$. Let $A$ and $D$ be the other two corners adjacent to $B$ and $C$, respectively. $B$ and $C$ must be right corners, otherwise points inside of $[AB]$ or $[CD]$ would be at a greater distance from $x_0$. 
The rectangle $ABCD$ must be included inside of $\P$: if a part of $ABCD$ was outside $\P$, since $\P$ is simply connected, the minimal orthogonal path from $x_0$ to $B$ would need to go around one side of the missing part, and reaching the other side from $x_0$ would require a path with more turns.

We may assume w.l.o.g that $AB\leq CD$. Then, $A$ must be a left corner, else reaching $[AB]$ from $x_0$ would require one more turn than $[BC]$. If $AB=CD$, $D$ must similarly be a left turn, and we have the second motif. If $AB<CD$, we have the third motif.
\end{proof}

We may now prove \Cref{SC}.

\begin{proof}[Proof of \Cref{SC}]
Let $\A$, $\A'$ and $\M$ satisfy the conditions of \Cref{SC}; let $\T$ satisfy $(i)$ and let $m=\max t_{ik}$. It follows from applying \Cref{shadok} to a connected component of $\P_m(\M)$ that there is a rectangle $ABCD$ corresponding to one of the three motifs described in \Cref{pans} included inside $\P_m(\M)$. Since $m$ is the maximum coefficient, all cells in $\P_m(\M)$ have coefficient $m$. It follows from $(iii)$ and \Cref{touching polyo} that all cells orthogonally or diagonally adjacent to $\P_m(\M)$ have coefficient $m-1$. From \Cref{m_ij}, we have $m_{ij}=t_{ij}+t_{i-1,j-1}-t_{i,j-1}-t_{i-1,j}$, where $t_{ij}=0$ if $i=0$ or $p$ or $j=0$ or $q$. For a right (resp. left) corner of $\P_m(\M)$, three (resp. one) of the incident cells have coefficient $m-1$ and one (resp. three) has coefficient $m$. Thus coefficients of $\M$ located at corners of $\P_m(\M)$ must be $1$ or $-1$. More precisely, regardless of the orientation of $ABCD$, the coefficients of $\M$ located at the corners are $1$ for the top left and bottom right corners and $-1$ for the top right and bottom left, except for $D$ in motif 3 which has a $0$.

A $1$ in $\M$ means a $0$ in $\A$ and $1$ in $\A'$; a $-1$ in $\M$ is the reverse; and a $0$ in $\M$ means either two $0$s or two $1$s in $\A$ and $\A'$. So, a $\begin{pmatrix}
1 & -1\\
-1 & 1
\end{pmatrix}$ sub-matrix in $\M$ means there is a negative checkerboard in $\A$ and a positive checkerboard in $\A'$. If any one of the four coefficients is replaced by $0$, then there is either a negative checkerboard in $\A$ or a positive checkerboard in $\A'$. Switching that checkerboard, either in $\A$ or in $\A'$, results in a new instance where $ABCD$ has been cropped from $\P_m(\M)$. Thus, the new instance still satisfies the three conditions of \Cref{SC}. Repeating the process terminates with $\M=0$, as $\sum t_{ik}$ is reduced at each step. This constructs a directed path from $\A$ to $\A'$ in $G(R,C)$.
\end{proof}

\bibliographystyle{plain}
\bibliography{references}

\end{document}